\documentclass[11pt]{preprint}

\usepackage{amsmath}
\usepackage{mhequ}
\usepackage{amssymb}
\usepackage{amsthm}
\usepackage{enumerate}
\usepackage{times}
\usepackage{microtype}

\theoremstyle{plain}
\newtheorem{theorem}{Theorem}[section]

\newtheorem{corollary}[theorem]{Corollary}
\newtheorem{lemma}[theorem]{Lemma}
\newtheorem{proposition}[theorem]{Proposition}

\theoremstyle{definition}
\newtheorem{remark}[theorem]{Remark}

\newcommand{\cc}[1]{\mathcal{#1}}

\renewcommand{\P}{\mathbb{P}}

\newcommand{\E}{\mathbb{E}}

\newcommand{\R}{\mathbb{R}}

\newcommand{\ET}{\mathcal{E}_T}
\newcommand{\CT}{\CC([0,T]; \CC(\T^2))}
\newcommand{\ETw}{\mathcal{E}_{\tau_\eps}}
\newcommand{\T}{\mathbb{T}}

\newcommand{\Z}{\mathbb{Z}}

\newcommand{\eps}{\varepsilon}

\newcommand{\N}{\mathbb{N}}

\newcommand{\bes}[3]{\mathcal{B}_{#1, #2}^{#3}}
\newcommand{\besn}[0]{\mathcal{B}_{p, r}^{s}}

\newcommand{\filt}[2]{\cc F_{#1}^{#2}} 

\newcommand{\set}[1]{\left\{#1\right\}} 
\newcommand{\ob}[1]{\left(#1\right)} 
\newcommand{\cb}[1]{\left[#1\right]} 
\newcommand{\abs}[1]{\left\vert#1\right\vert} 
\newcommand{\norm}[1]{\left\|#1\right\|}

\newcommand{\wick}[1]{{:}#1{:}}

\def\CC{\mathcal{C}}
\def\CM{\mathcal{M}}

\def\eref#1{(\ref{#1})}
\def\d{\partial}

\title{Triviality of the 2D stochastic Allen-Cahn equation}
\begin{document}
\author{M. Hairer$^1$, M.D. Ryser$^2$, and H. Weber$^3$}
\institute{University of Warwick, \email{M.Hairer@Warwick.ac.uk}
\and Duke University, \email{ryser@math.duke.edu}
\and University of Warwick, \email{H.Weber@warwick.ac.uk}}

\maketitle

\begin{abstract}
We consider the stochastic Allen-Cahn equation driven by mollified
space-time white noise. We show that, as the mollifier is removed, the 
solutions converge weakly to $0$, independently of the initial condition.
If the intensity of the noise simultaneously converges to $0$ at a sufficiently fast
rate, then the solutions converge to those of the deterministic equation.
At the critical rate, the limiting solution is still deterministic, but it exhibits
an additional damping term.
\end{abstract}

\section{Introduction}\label{intro}

We consider the following evolution equation on the two-dimensional torus $\T^2$:
\begin{equ}[e1]
du= \ob{\Delta u +u -u^3} dt + \sigma dW\;,\quad u(0)=u^0\;. \tag{$\Phi$}
\end{equ}
Here $u^0$ is a suitably regular initial condition, $\sigma$ a positive constant, and $W$ an $L^2(\T^2)$-valued cylindrical Wiener process defined on a probability space $(\Omega, \filt{}{}, \P)$. In other words, at least at a formal level, ${dW\over dt}$ is space-time white noise.

This equation and variants thereof have a long history. The deterministic part of the equation is the $L^2$
gradient flow of the Ginzburg-Landau free energy functional
\begin{equ}
\int_{\T^2}\Bigl({1\over 2}|\nabla u(x)|^2 + V(u(x))\Bigr)\,dx\;,
\end{equ}
with the potential energy $V$ given by the standard double-well function $V(u) = {1\over 4}(u^2-1)^2$, 
see \cite{GL}. This provides a phenomenological model for the evolution of an order
parameter describing phase coexistence in a system without preservation of mass.
At large scales, the dynamic of phase boundaries is known to converge to the mean curvature flow \cite{allen1979microscopic,MeanCurv,MeanCurv2}.

The noise term $\sigma dW$ accounts for thermal fluctuations at positive temperature. On a formal level the choice of space-time white noise is natural, because it satisfies the right fluctuation-dissipation relation. At least for finite-dimensional gradient flows it is natural to take the bilinear form that determines the mechanism of energy dissipation as covariance of the noise, as this guarantees the invariance of the right Gibbs measure under the dynamics. Naively extending this observation to the current infinite dimensional context yields \eqref{e1}. 

White noise driven equations such as \eqref{e1} are known to be ill-posed in space-dimension $d \geq2$ \cite{walsh1986introduction,da1992stochastic}.  Actually,  the linearised version of \eqref{e1} (simply remove the term $u^3$) admits only distribution-valued solutions for $d \ge 2$.  For any $\kappa>0$ these solutions take values in the Sobolev space $H^{  \frac{2 - d}{2} - \kappa}$, but they do not take values in $H^{\frac{2 - d}{2}}$. In general, it is impossible to apply nonlinear functions to elements of these spaces and the standard approach to construct solutions of \eqref{e1} \cite{da1992stochastic, hairer2009introduction}  fails.

In the present article, we introduce a cutoff at spatial lengths of order $\eps$ and we
study the limit as $\eps \to 0$ for finite noise strength for \eqref{e1}. More precisely, we set
\begin{equ}
W_\eps(t)=\sum_{|k| \leq 1/\epsilon} e_k \beta_k(t), \qquad \eps>0\;,
\end{equ}
where $\set{e_k}_{k\in\Z^2}$ is the Fourier basis on $\T^2$, and $\set{\beta_k}_{k\in\Z^2}$ are 
complex Brownian motions that are i.i.d.\ except for the reality condition $\bar \beta_k = \beta_{-k}$. 
We thus consider
\begin{equ}[e103]
du_\eps= \ob{\Delta u_\eps +u_\eps -u_\eps^3} dt + \sigma(\eps)\, dW_\eps\;,\quad 
u_\eps(0)=u^0\;, \tag{${\Phi}_\eps$}
\end{equ}
and study the weak limit of $u_\eps$ as $\eps \to 0$.

The main result of this article can loosely be formulated as follows (a precise statement  will be given in Theorems~\ref{main_thm1} and \ref{main_thm2} below):
\begin{theorem}\label{main_main_thm_1}
Let $\sigma$ be bounded and such that $\lim_{\eps\to0}\sigma^2(\eps) \log 1/\eps =  \lambda^2 \in [0,+\infty]$.
If $\lambda^2 = +\infty$, then $u_\eps$ converges weakly to $0$, in probability. Otherwise, 
it converges weakly in probability to the solution $w_\lambda$ of
\begin{equ}\label{det_alt}
\partial_t w_\lambda  = \Delta w_\lambda - \ob{{\textstyle\frac{3}{8\pi}}\lambda^2-1} w_\lambda - w_\lambda^3\;,\quad
w_\lambda(0) =u^0\;. \tag{$\Psi_\lambda$}
\end{equ}
\end{theorem}

\begin{remark}
The result for constant $\sigma$ was conjectured in \cite{ryser2011well}, based on numerical simulations.
\end{remark}

\begin{remark}
The borderline case $\lambda \neq 0$ is particularly interesting as it provides an example of stochastic damping: in the limit as $\eps \to 0$, the stochastic forcing is converted into an additional deterministic damping term, $-\frac{3}{8\pi}\lambda^2 w_\lambda$, to the Allen-Cahn equation. 
In particular, if $\lambda^2 > {8\pi\over3}$, the zero-solution becomes globally attracting.   
\end{remark}

\begin{remark}
Recently, there has been a lot of interest in \eqref{e1} in the regime where the noise is small \cite{kohn2007action, BovierAC10, Cerr11}. There, the authors studied \eqref{e1} in arbitrary space dimension on the level of large deviation theory. As in \eqref{e103} they consider a modified version of \eqref{e1} where the noise term $dW$ is replaced by a noise term $dW_\eps$ with a finite spatial correlation length $\eps$. For this modified equation, solutions can be constructed in a standard way and a large deviation principle \`a la Freidlin-Wentzell can be obtained. 
One can then show that the rate functionals converge as $\eps \to 0$. 
The large deviation principle however is \textit{not} uniform in $\eps$;
this procedure corresponds to taking the amplitude of the noise much smaller than $\eps$.
The results obtained in this article quantify how small the noise should be as a function of $\eps$ in order for the 
solutions of \eqref{e1} to be close to the deterministic equation.
\end{remark}

\begin{remark}
We believe that the weak convergence to $0$ as $\eps \to 0$ actually holds for a much larger class
of potentials. Actually, one would expect it to be true whenever $\lim_{|u| \to \infty} V''(u) = +\infty$.
The proof given in this article does however depend crucially on the fact that $V(u) \sim u^4$ for
large values of $u$.
\end{remark}

The main tools used in our proofs are provided by the theory of stochastic quantisation. Actually, in the context of Euclidean Quantum Field Theory the question of existence  of the formal invariant measure of \eqref{e1} has been treated in the seventies (see e.g. \cite{GliJaf87}). Then, it had been observed that this measure, the so called $\Phi^4_2$ field, can be defined, but only if a logarithmically diverging lower order term is subtracted. The corresponding stochastic dynamical system (i.e. the renormalised version of \eqref{e1}) has also been constructed \cite{ParisiWu81, AlbRock91, da2003strong}. Note that although this renormalised equation,
\begin{equ}
du= \ob{\Delta u +u -\wick{u^3}} dt + \sigma dW\;,
\end{equ}
formally resembles \eqref{e1} it does not have a natural interpretation as a phase field model.

Our main argument is a modification of the construction provided in \cite{da2003strong}. We present here a brief heuristic argument for the case $\sigma \equiv 1$. First, let $C_\eps>1$ and add and subtract the term $C_\eps u_\eps$ to (\ref{e103}) to get
\begin{equ}\label{e3hu}
du_\eps= \ob{\Delta u_\eps - \ob{C_\eps -1} u_\eps -u_\eps \ob{u_\eps^2- C_\eps}} dt + dW_\eps\;.
\end{equ}
The key idea is to choose $C_\eps$ in such a way that, for small values of $\eps$, the term
$u_\eps \ob{u_\eps^2 - C_\eps}$ is equal to the 
Wick product $\wick{u_\eps^3}$ with respect to the Gaussian structure given by the invariant measure of the linearised system (which itself depends on $C_\eps$).
Since, given the results in \cite{da2003strong}, one would expect $\wick{u_\eps^3}$ 
to at least remain bounded as $\eps \to 0$, it is  not surprising that
the additional strong damping term $-C_\eps u_\eps$ causes the solution to vanish in the limit.

\subsection*{Acknowledgements}

{\small
MH acknowledges financial support by the EPSRC trough grant EP/D071593/1 and the 
Royal Society through a Wolfson Research Merit Award. Both MH and HW were supported by the Leverhulme Trust through 
a Philip Leverhulme Prize. MDR is grateful to P.F. Tupper and N. Nigam for fruitful discussions, and acknowledges financial support from a Hydro-Qu\'{e}bec Doctoral Fellowship. All three authors are grateful for the relaxed atmosphere at the Newton institute, where this 
collaboration was initiated during the 2010 ``Stochastic PDEs'' programme.
}

\section{Notations and Main Result}\label{main}

In order to formulate our results, we first introduce the class of Besov spaces that we will work with. As in \cite{da2003strong} we choose to work in Besov spaces, because they satisfy the right multiplicative inequalities (see Lemma \ref{cor10}). Denote by $\ob{\cdot, \cdot}$ the $L^2$ inner product, and  by $\set{e_k(x)=\frac{1}{2\pi} e^{ikx}}_{k \in \Z^2}$ the corresponding orthonormal Fourier basis. Throughout the article, we work with periodic Besov spaces $\besn(\T^2)$, where $p,r \geq 1$ and $s\in \R$. These spaces are defined as the closure of $C^\infty(\T^2)$ under the norm
\begin{equ}\label{e3abc}
\norm{u}_{\bes{ p}{r}{s}(\T^2)}:= \Bigl( \sum_{q=0}^\infty 2^{qrs} \, \norm{\Delta_q u}_{L^p(\T^2)}^r\Bigr)^{1/r},
\end{equ}
where the $\Delta_q$ are the Littlewood-Paley projection operators given by $\Delta_0 u= \ob{e_0, u} e_0$ and
\[ \Delta_q u=\sum_{2^{q-1}\leq \abs{k} < 2^q} \ob{e_k,u} e_k, \qquad q\geq1. \] 
Regarding the exponents appearing in these Besov spaces, we will restrict ourselves throughout this article
to exponents $p$, $r$ and $s$ such that
\begin{equ}[e:assExp1]
p\ge 4\;,\quad  r\geq 1\;,\quad
-\frac{2}{7p} < s <0\;.
\end{equ}

We now reformulate Theorem~\ref{main_main_thm_1} more precisely. The case $\lambda^2 = +\infty$ is given
by the following:

%

\begin{theorem}\label{main_thm1}
Assume   $u^0\in \besn$ such that (\ref{e:assExp1}) holds. 
Then for all $\eps>0$ and $T>0$, there exists a unique mild solution $u_\eps$.
If $\sigma(\eps)$ is bounded uniformly in $\eps$ and satisfies $\lim_{\eps\to 0} \sigma^2(\eps) \log(1/\eps) =+\infty$ 
then, for all $\delta\in(0,T)$, 
$\lim_{\eps\to 0} \norm{u_\eps}_{\CC([\delta ,T];\besn)}= 0$ in probability.
\end{theorem}

On the other hand, when $\sigma^2(\eps) \log(1/\eps)$ converges to a finite limit, we have

\begin{theorem}\label{main_thm2}
Assume   $u^0\in \besn$ such that (\ref{e:assExp1}) holds. 
If $\lim_{\eps\to 0} \sigma^2(\eps) \log(1/\eps) = \lambda^2 \in \R$ 
then, for all $\delta\in(0,T)$, 
$\lim_{\eps\to 0} \norm{u_\eps - w_\lambda}_{\CC([\delta ,T];\besn)}= 0$ in probability,
where $w_\lambda$ is the unique solution to $($\ref{det_alt}$)$.
\end{theorem}

 \begin{remark}
If $\sigma$ decays sufficiently fast, for example $\sigma(\eps)\sim \eps^{\tau}$ for some $\tau>0$, then the conclusion of Theorem \ref{main_thm2} actually holds in the space of space-time continuous functions.
 \end{remark}

To conclude this section, we introduce some concepts borrowed from the theory of stochastic quantization. Since we are not concerned with the dynamics of quantised fields, we only introduce the notions necessary for the proof techniques used below, and refer to \cite{da2007wick} for a general introduction to the topic. Consider the linear version of equation (\ref{e3hu}), namely
\begin{equ}\label{eee123}
dz_\eps= \ob{\Delta z_\eps - \ob{C_\eps -1} z_\eps} dt + \sigma(\eps)dW_\eps\;.
\end{equ}
For $C_\eps > 1$, this equation has a unique invariant measure on $L^2(\T^2)$, which we denote by $\mu_\eps$.
It is $\mu_\eps$ that will play the role of the ``free field'' in the present article.
Under $\mu_\eps$, the $k$th Fourier component of $z_\eps$ 
is a centred complex Gaussian random variable with variance ${\sigma^2(\eps)\over 2}\ob{C_\eps -1 +|k|^2}^{-1}$. 
Furthermore, distinct Fourier components are independent, except for the reality condition
$z_\eps(-k) = \overline{z_\eps(k)}$.

As a consequence of translation invariance, one has the identity
\begin{equs}\label{rho}
D_\eps^2 &:= \int_{L^2} |\phi_\eps(x)|^2\mu_\eps(d\phi_\eps) = \ob{\frac{1}{2\pi}}^2 \int_{L^2} \norm{\phi_\eps}_{L^2}^2 \mu_\eps(d\phi_\eps)\\
&= \frac{1}{8\pi^2} {\sum_{\abs{k}\leq 1/\eps} \frac{\sigma^2(\eps)}{C_\eps -1 +\abs{k}^2}}\;.
\end{equs}
We then \textit{define} the 
Wick powers of any field $u_\eps$ with respect to the Gaussian structure given by $\mu_\eps$
by
\begin{equ}
\wick{u_\eps^n} = D_\eps^n H_n(u_\eps/D_\eps)\;,
\end{equ}
where $H_n$ denotes the $n$th Hermite polynomial. In this article, we will only ever use the Wick
powers for $n \le 3$, for which one has the identities
\begin{equ}\label{wickpower}
\wick{u_\eps^1}=u_\eps\;,\qquad \wick{u_\eps^2}= u_\eps^2-D_\eps^2\;,\qquad
\wick{u_\eps^3}=u_\eps^3-3\,D_\eps^2\,u_\eps\;.
\end{equ} 
From now on, whenever we use the notation $\wick{u_\eps^n}$, (\ref{wickpower}) is what we refer to.
For any two expressions $A$ and $B$ depending on $\eps$, we will throughout this article use 
the notation $A \lesssim B$ to
mean that there exists a constant $C$ independent of $\eps$ (and possibly 
of other relevant parameters clear from
the respective contexts) such that $A \le C\, B$. 

\section{Trivial limit for strong noise}\label{proof1}
 
In this section, we provide the proof of Theorem~\ref{main_thm1}. First, in Subsection~\ref{step1}, we 
explain the ``correct'' choice of the renormalization constant $C_\eps$ in (\ref{e3hu}). In Subsection~\ref{step2},
we then obtain bounds on the linearised equation, as well as its Wick powers. Finally,
in Subsection~\ref{step3}, we obtain a bound on the remainder and we combine these results in order to conclude.
 
\subsection{Fixing the renormalization constant}\label{step1}
For $C_\eps>1$, we rewrite (\ref{e103}) as  
\begin{equ}[e21]
du_\eps= \ob{A_\eps u_\eps -u_\eps(u_\eps^2-C_\eps)} dt +\sigma(\eps)dW_\eps\;,
\end{equ}
where the linear operator $A_\eps$ is given by $ A_\eps = \Delta  -(C_\eps -1)$.
Motivated by the heuristic arguments provided in Section \ref{intro}, the goal of this section is to determine $C_\eps$ in such a way that the nonlinear term $u_\eps(u_\eps^2-C_\eps)$ is equal to the Wick product $\wick{u_\eps^3}$. 
It then follows from 
 (\ref{rho}) and  (\ref{wickpower}) that $C_\eps$ is implicitly determined by the equation
\begin{equ}\label{e11}
C_\eps= 3D_\eps^2 = \frac{3}{8\pi^2} \, \,\sum_{|k|\leq 1/\eps} \frac{\sigma^2(\eps)}{C_\eps - 1 + \abs{k}^2}\;.
\end{equ}
To describe the behavior of the solution to (\ref{e11}), we shall use the notation $A_\eps \sim B_\eps$ to mean $\lim_{\eps\to 0} A_\eps/B_\eps =1$.

\begin{lemma}\label{lemma1}
For any values of the parameters, equation (\ref{e11}) has a unique solution $C_\eps > 1$. 
If $\sigma$ is uniformly bounded and such that $\lim_{\eps \to 0} \sigma^2(\eps)\log (1/\eps) = \infty$,
then one has  
\begin{equ}\label{e12}
C_\eps \sim 
\frac{3}{4\pi}\, \sigma^2(\eps)\log{\frac{1}{\eps}}\;.
\end{equ}
In particular, $\lim_{\eps\to0} C_\eps = +\infty$.
\end{lemma}

Before we proceed to the proof of this result, we state the following very useful result:

\begin{lemma}\label{lem:intLog}
Let $a, R \ge 1$. Then there exists a constant $C$ such that the bound
\begin{equ}[e:RHS]
\Bigl| \sum_{|k| \le R} {1\over a+|k|^2} - \pi \log \Bigl( 1 + {R^2 \over a}\Bigr)\Bigr| \le {C\over \sqrt a} \Bigl(1 \wedge {R \over \sqrt a}\Bigr) \;,
\end{equ}
holds. Here, the sum goes over elements $k \in \Z^2$.
\end{lemma}

\begin{proof}
The second expression on the left is nothing but $\int_{|k| \le R} {dk \over a+|k|^2}$, so we
want to bound the difference between the sum and the integral. Using the monotonicity and positivity of the
function $x \mapsto {1\over a+x^2}$ and restricting ourselves to one quadrant, 
we see that one has the bounds
\begin{equ}
 \sum_{|k| \le R \atop k_i > 0} {1\over a+|k|^2} \le {1\over 4}\int_{|k| \le R} {dk \over a+|k|^2} \le \sum_{|k| \le R \atop k_i \ge 0} {1\over a+|k|^2}\;.
\end{equ}
As a consequence, the required error is bounded by
\begin{equ}
4 \sum_{k = 0}^{\lfloor R\rfloor}  {1\over a+k^2}\le {4\over a} + 4 \int_0^R {dx \over a+x^2} \;.
\end{equ}
The required bound follows at once, using the fact that $a$ and $R$ are bounded away from $0$ by assumption.
\end{proof}

\begin{proof}[Proof of Lemma~\ref{lemma1}]
Since  the right hand side decreases from $\infty$ down to $0$ as the left hand side grows from $1$ to $\infty$, it follows immediately that (\ref{e11}) always has a unique solution $C_\eps > 1$.

Since, by Lemma~\ref{lem:intLog}, one has $\sum_{|k| \le {1\over \eps}} \frac{1}{1+ |k|^2} \sim 2\pi \log{1\over \eps}$
and since by assumption $\sigma^2(\eps) \log{1\over \eps} \to \infty$, there exists $\eps_0$ such that
\begin{equ}[e:simpleBound]
\frac{3}{8\pi^2} \sum_{|k|\leq 1/\eps} \frac{\sigma^2(\eps)}{1 + \abs{k}^2} \ge 2\;,
\end{equ}
for all $\eps < \eps_0$. As a consequence, we have $C_\eps \ge 2$ for such values of $\eps$,
and we will use this bound from now on. On the other hand, if we know that 
$C_\eps \ge 2$, then $C_\eps$ is  bounded
\emph{from above} by the left hand side of (\ref{e:simpleBound}), so that
\begin{equ}[e:upperBound]
C_\eps \le K \sigma^2(\eps) \log{1\over \eps}\;,
\end{equ}
for some constant $K$ and for $\eps$ small enough.

It now follows from Lemma~\ref{lem:intLog} that
\begin{equ}[e:equalityCeps]
C_\eps =  {3 \sigma^2(\eps) \over 8\pi} \log{\ob{1+\frac{1}{\eps^2\ob{C_\eps -1}}}}+ R_\eps\;,
\end{equ}
for some remainder $R_\eps$ which is uniformly bounded as $\eps \to 0$.
Since, by \eref{e:upperBound}, the first term on the right hand side goes to $\infty$, this 
shows that $R_\eps$ is negligible in (\ref{e:equalityCeps}), so that
\begin{equ}
C_\eps \sim {3 \sigma^2(\eps) \over 8\pi} \log{\ob{\frac{1}{\eps^2 C_\eps}}}
= {3 \sigma^2(\eps) \over 8\pi} \Bigl( \log{1\over \eps^2} - \log C_\eps\Bigr)\;.
\end{equ}
Since $C_\eps$ is negligible with respect to $1\over \eps^2$ by \eref{e:upperBound}, the claim follows.
\end{proof}
\subsection{Bounds on the linearised equation}\label{step2}

We split the solution to (\ref{e21}) into two parts by introducing the stochastic convolution
\begin{equ}\label{e23}
z_\eps(t):=\sigma(\eps)\int_{-\infty}^t e^{(t-s)A_\eps} dW_\eps(s)\;,
\end{equ}
and performing the change of variables $v_\eps(t):= u_\eps(t) - z_\eps(t)$. With these notations, $v_\eps$
solves
\begin{equs}\label{e26} 
\partial_t v_\eps& = A_\eps v_\eps - \ob{v_\eps^3  + 3 v_\eps^2  z_\eps +3 v_\eps \wick{z_\eps^2} + \wick{z_\eps^3}}\tag{$\Phi_\eps^{{aux}}$}\\
v_\eps(0)&= u^0-z_\eps(0).
\end{equs}
We thus split the original problem into two parts: first, we show that the stochastic convolution 
converges to $0$, then we show that the remainder $v_\eps$ also converges to $0$.

By construction, the stochastic convolution (\ref{e23}) is a stationary process and its invariant measure is given by $\mu_\eps$. We first establish a general estimate for its renormalized powers $\wick{z_\eps^n}$,
which will be useful for bounding $v_\eps$ later on. Throughout this section, we assume that 
$\lim_{\eps\to 0}\sigma^2(\eps)\log(1/\eps) = \infty$ and that $C_\eps$ is given by \eref{e11}.
We then have:

\begin{lemma}\label{lemma2}
Let $r,k,p\geq1$, $s<0$. Then, for all $n\in\mathbb{N}$, we have 
\begin{align}\label{e32}
\lim_{\eps\to 0}\, \E \,\norm{\wick{z_\eps^n}}_{\bes{p}{r}{s}}^k = 0\;. 
\end{align}
\end{lemma}

\begin{proof} 
Following the calculations of the proof of \cite[Lemma 3.2]{da2003strong},
we see that 
\begin{equ}\label{tildebound}
\E \,\norm{\wick{z_\eps^n}}_{\bes{p}{r}{s}}^k  \lesssim \|\gamma_\eps\|^{kn\over 2}_{H^{\beta_n}}\;, 
\end{equ}
where $\beta_n = 1 + {r k s\over 2^{n}p}$ and 
\begin{equ}
\gamma_\eps(x) = \sum_{|k| \le 1/\eps} \frac{\sigma^2(\eps)}{C_\eps -1 + \abs{k}^2} e_k(x)\;.
\end{equ}
Since
\begin{equ}
\|\gamma_\eps\|^{2}_{H^{\beta_n}} =  \sum_{|k| \le 1/\eps} \frac{\sigma^4(\eps)(1+|k|^2)^{\beta_n}}{(C_\eps -1 + \abs{k}^2)^2}\;,
\end{equ}
and $\beta_n < 1$, the claim follows from the boundedness of $\sigma$ and the fact that $C_\eps \to \infty$.
\end{proof}

\begin{corollary}\label{aha2}
Let $n,p,r\geq 1$ and $s<0$. Then $\wick{z_\eps^n} \in L^p([0,T]; \bes{p}{r}{s})$ $\P$-a.s., for all $\eps>0$.
In particular,
\begin{align}\label{e59n}
\lim_{\eps\to0} \E \norm{\wick{z_\eps^n}}_{L^p(0,T; \besn)} = 0.
\end{align}
\end{corollary}
\begin{proof} 
This follows from the stationarity of $z_\eps$, Fubini's theorem and Lemma \ref{lemma2}. 
\end{proof}

\noindent We establish now the main result of this subsection.
\begin{proposition}\label{prop1}
Consider the stochastic convolution $z_\eps$  defined in (\ref{e23}) and let $p,r\geq 1$, $s<0$ and $T>0$. Then
\begin{equ}\label{e27}
\lim_{\eps \to 0} \E \norm{z_\eps}_{\CC([0,T]; \bes{p}{r}{s})}  = 0\;.
\end{equ}
\end{proposition}

\begin{proof}
We begin by decomposing the stochastic convolution into two parts,
\begin{align*}
z_\eps(t)= e^{tA_\eps} z_\eps(0)+ \sigma(\eps)\int_{0}^t\, e^{\ob{t-s}A_\eps}dW_\eps(s).
\end{align*}
The bound on the first term follows from Lemma~\ref{lemma2} and Lemma~\ref{a14},
so it remains to focus on the second term, which we denote hereafter as
$\bar{z}_\eps(t)$.
In order to bound it, we use the {\it factorization method}, see \cite[p.~128]{da1992stochastic}, as well as  \cite[p.~47]{hairer2009introduction} for a more detailed presentation. Recalling that
\begin{equ}
\int_\sigma^t \ob{t-s}^{\alpha-1} \ob{s-\sigma}^{-\alpha}ds = \frac{\pi}{\sin \pi\alpha}\;,
\end{equ}
we fix $\alpha\in(0,\frac12)$ and rewrite $\bar{z}_\eps$ as
\begin{align}\label{e27p1b}
\bar{z}_\eps(t)= \frac{\sin \pi \alpha}{\pi} \int_0^t e^{(t-s)A_\eps}\,Y_\eps(s) \,(t-s)^{\alpha-1}ds,
\end{align}
where 
\begin{align*}
Y_\eps(s):=\sigma(\eps) \int_0^s (s-\sigma)^{-\alpha}\,e^{(s-\sigma)A_\eps}  dW_\eps(\sigma).
\end{align*}
Next, we introduce the mapping $\Gamma_\eps: y \mapsto \Gamma_\eps y$ defined by \[\Gamma_\eps y(t): =\frac{\sin \pi \alpha}{\pi} \int_0^t e^{(t-s)A_\eps}\,y(s) \,(t-s)^{\alpha-1}ds,\] and show that $\Gamma_\eps: L^q([0,T]; \bes{p}{r}{s}) \to \CC([0,T]; \bes{p}{r}{s})$ is a bounded mapping for $q>1/\alpha$.  
First, it is a consequence of the strong continuity of $e^{tA_\eps}$  that $\Gamma_\eps \,y \in \CC([0,T]; \besn)$ for all  $y\in \CC([0,T]; \besn)$ such that $y(0)=0$  \cite[p.~48]{hairer2009introduction}. Next, observe that $s\mapsto  (t-s)^{\alpha-1}$ is in $L^{\bar q}(\cb{0,t})$ for all $\bar q\in[1,(1-\alpha)^{-1})$, and hence we can use H\"{o}lder's inequality to deduce that for all $q>\frac{1}{\alpha}$,
\begin{align}\label{e27p1}
\sup_{t\in[0,T]} \norm{\Gamma_\eps y \,(t)}_{\bes{p}{r}{s}}\lesssim \norm{y}_{L^q(\cb{0,T};\bes{p}{r}{s})}.
\end{align}
A standard density argument allows us to conclude that $\Gamma_\eps: L^q([0,T]; \bes{p}{r}{s}) \to \CC([0,T]; \bes{p}{r}{s})$ is indeed a bounded mapping for $q>1/\alpha$.   

To conclude the proof, we assume for the moment that there exist $K_\eps>0$ such that
\begin{equ}\label{e27p3}
\sup_{t\in \cb{0,T}} \E \norm{Y_\eps(t)}_{\besn} \leq K_\eps\;, \qquad \lim_{\eps\to 0} K_\eps =0\;.
\end{equ}
From (\ref{e27p3}), it then follows that  
\begin{align}\label{e27p3b}
\E \norm{Y_\eps}_{L^q(\cb{0,T}; \besn)} \leq \Bigl(T \sup_{t\in \cb{0,T}}\E \norm{Y_\eps}_{\besn}^q\Bigr)^{1/q}\lesssim T^{1/q} K_\eps\;,
\end{align}
where the first inequality is due to Jensen's inequality and Fubini's theorem, and the second inequality follows from (\ref{e27p3}) in conjunction with Fernique's theorem. By (\ref{e27p3b}), $Y_\eps \in L^q(\cb{0,T}; \besn)$ $\P$-a.s.\ and hence $\bar{z}_\eps = \Gamma_\eps Y_\eps \in \CC([0,T]; \besn)$ $\P$-a.s. Furthermore, it follows from (\ref{e27p1})--(\ref{e27p3b}) that 
\begin{align*}
\E \,\sup_{t\in [0,T]} \norm{\bar{z}_\eps(t)}_{\besn}& \lesssim \E \norm{Y_\eps}_{L^q(\cb{0,T}; \besn)}
 \lesssim K_\eps\;, 
\end{align*}
so that  $\norm{\bar{z}_\eps}_{\CC([0,T]; \bes{p}{r}{s})} \to 0$ in probability, as required. 

It remains to establish (\ref{e27p3}).
By definition of the Besov norm (\ref{e3abc}) and Jensen's inequality,
\begin{equ}\label{e27p3a}
\E \norm{Y_\eps(t)}_{\besn} \leq  \Bigl(\sum_{q=0}^\infty 2^{qrs}\, \E \norm{\Delta_q Y_\eps(t)}_{L^p}^r\Bigr)^{1/r}.
\end{equ} 
As a consequence, (\ref{e27p3}) follows if we can show that
\begin{equ}[e:goal]
\E \norm{\Delta_q Y_\eps(t)}_{L^p}^p \le K_\eps 2^{qp\tau}\;,
\end{equ}
for some $\tau < |s|$ and some $K_\eps \to 0$.

Fix now $q\in \N$. Thanks to Fubini's theorem, the Gaussianity of  $\Delta_q Y_\eps(t)$, 
and the independence of its different Fourier components, 
\begin{align}\label{e27p4}
\E \norm{\Delta_q Y_\eps(t)}_{L^p}^p &= \int_{\T^2}{\E \Bigl|\sum_{2^{q-1}\leq \abs{k}<2^q} \ob{Y_\eps(t), e_k}e_k(\xi)} \Bigr|^p\,d\xi \nonumber\\
&\lesssim \int_{\T^2}{\Bigl(\E \Bigl|\sum_{2^{q-1}\leq \abs{k}<2^q} \ob{Y_\eps(t), e_k}e_k(\xi)} \Bigr|^2\Bigr)^{p/2}\,d\xi\\
&\lesssim \int_{\T^2} \Bigl(\sum_{2^{q-1}\leq \abs{k}<2^q} \E \abs{\ob{Y_\eps(t), e_k}}^2\Bigr)^{p/2}\, d\xi.\nonumber
\end{align}
It\^{o}'s isometry and the definition of $A_\eps$ yield
\begin{align}\label{starbound}
\E \abs{\ob{Y_\eps(t), e_k}}^2&
\leq   \sigma^2(\eps)\cb{2\ob{C_\eps-1+\abs{k}^2}}^{2\alpha-1}\, \int_0^\infty  e^{-\tau} \tau^{-2\alpha} d\tau\nonumber\\
& \lesssim \sigma^2(\eps) \ob{C_\eps-1+\abs{k}^2}^{2\alpha-1},
\end{align}
where the last inequality is due to $2\alpha<1$. Inserting (\ref{starbound}) back into (\ref{e27p4}) we obtain
the bound
\begin{equs}
\E \norm{\Delta_q Y_\eps(t)}_{L^p}^p  &\lesssim \sigma^p(\eps)\ob{\sum_{2^{q-1} \leq \abs{k}<2^q} \ob{\frac{1}{C_\eps-1+\abs{k}^2}}^{1-2\alpha}}^{p/2}\\
&\lesssim\sigma^p(\eps)\ob{ {2^{2q\tau} \over (C_\eps-1)^\delta} \sum_{2^{q-1} \leq \abs{k}<2^q}\frac{1}{\abs{k}^{2+2\tau - 4\alpha-2\delta}}}^{p/2},
\end{equs}
which is valid for all $\tau>0$ and all $\delta \in (0, 1-2\alpha)$. 
Since we can make both $\alpha$ and $\delta$ arbitrarily small, we can in particular choose them
in such a way that $2\alpha + \delta < \tau < |s|$, so that the exponent is strictly greater than $2$.
This implies that the corresponding inverse power of $|k|$ is summable over all $k$, so that (\ref{e:goal}) is satisfied. 
\end{proof}

\subsection{Bounds on the remainder}\label{step3}

First, we need a technical lemma for the mapping $\CM^\eps$, defined as
\begin{align}\label{e58a}
\ob{\CM^\eps y}(t):= e^{t\,A_\eps} \ob{u^0-z_\eps(0)} + \int_0^t e^{(t-\tau)\, A_\eps} \, \sum_{l=0}^{3}\, a_l\,y^l(\tau) \, \wick{z_\eps^{3-l}(\tau)}\,d\tau\;,
\end{align}
where the $a_l$ are some real-valued constants.
In order to formulate the results of this section, we introduce the Banach space  
\[\ET:= \CC([0,T];\besn)\cap L^p([0,T];\bes{p}{r}{\bar s}),\] equipped with the usual maximum norm
 \begin{align}\label{max_norm}
\norm{x}_{\ET}:=\max\ob{\norm{x}_{C\ob{[0,T]; \bes{p}{r}{s}}}, \norm{x}_{L^p\ob{[0,T];\bes{p}{r}{\bar s}}}}.
\end{align}
Regarding the parameters appearing in $\ET$, we shall usually assume that
$(p,r,s,\bar s)$ satisfy the bounds
\begin{equ}[e:assExp]
p\ge 4\;,\quad  r\geq 1\;,\quad \bar s = 2s +\frac 2p\;,\quad
-\frac{2}{7p} < s <0\;.
\end{equ}

\begin{lemma}\label{e57}
Fix $\eps>0$, $T>0$, and assume \eref{e:assExp}. Then there exist positive constants $\delta$ and $K_\eps$ with $\lim_{\eps \to 0} K_\eps=0$ such that   
\begin{align}\label{e59f}
\norm{\CM^\eps y}_{\ET} \leq &\ob{1+ K_\eps\, T^{\delta}} \norm{u^0-z_\eps(0)}_{\bes{p}{r}{s}}\nonumber \\
& + K_\eps \, T^{\delta} \, \sum_{l=0}^{3} \bigl\|{\wick{z_\eps^{3-l}}}\bigr\|_{L^p([0,T];\bes{p}{r}{s})}\norm{y}^l_{\ET}.
\end{align}
\end{lemma}
\begin{proof}
The bound of the first term on the right-hand side of (\ref{e58a}) is given in Proposition~\ref{IC_cor}. Next, we split the second term into two parts, $\Omega_\eps^1+\Omega_\eps^2$, where
\begin{equs}
\Omega_\eps^1(t,y) &= \int_0^t e^{(t-\tau)\, A_\eps} \, \sum_{l=0}^{2}\, a_l \, \wick{z_\eps^{3-l}(\tau)} \,y^l(\tau) \,d\tau,\\
\Omega_\eps^2(t,y) &=\int_0^t e^{(t-\tau)\, A_\eps} \,  y^3(\tau) \, d\tau\;.
\end{equs}
We bound $\Omega_\eps^1$ first. Since $((l+1)-1)s +1 -2/p>0$ for $l=0,1,2$, we can employ Lemma~\ref{lemma_imp} to find that there exist $\delta > 0$ and $K_\eps$ as in the statement such that 
\begin{equ}
\norm{\int_0^t \, e^{(t-\tau)A_\eps} \wick{z_\eps^{3-l}(\tau)} \, y^l(\tau) d\tau}_{\ET} \leq K_\eps\, T^{\delta} \bigl\|{\wick{z_\eps^{3-l}}\, y^l}\bigr\|_{L^{p/(l+1)}([0,T]; \bes{p}{r}{(2l+1)s})}. 
\end{equ}
Using Lemma~\ref{cor10} and adding up the respective contributions yields the terms with $l = 0,1,2$ on the right-hand side of  (\ref{e59f}). 

We now bound  $\Omega_\eps^2$. Since $y\in\ET $ and $s<\bar s$, the embedding $\bes{p}{r}{\bar s} \hookrightarrow \besn$ implies that $y\in L^p([0,T];\besn)$. From  Lemma~\ref{lemma_imp} with  $n=3$ and  Lemma~\ref{cor10} with $l=2$, it follows again that there exist $\delta$ and  $K_\eps$ such that 
\begin{equs}\label{e3e}
\norm{\int_0^t e^{(t-\tau) A_\eps} y^3(\tau) d\tau}_{\ET} &\leq K_\eps \, T^{\delta} \, \norm{y\, y^2}_{L^{p/3}([0,T]; \bes{p}{r}{5s})} \\ &\leq K_\eps T^{\delta}\, \norm{y}^3_{L^p([0,T]; \bes{p}{r}{\bar s})}, 
\end{equs} 
which is the term with $l=3$ on the right-hand side of (\ref{e59f}).
\end{proof}

\begin{lemma}\label{huhu}
Let $\eps>0$, assume \eref{e:assExp} and consider $($\ref{e26}$)$ with $u^0\in \besn$. Then for all  $T>0$, there exists $\P$-a.s.\ a unique mild solution $v_\eps \in \ET$. 
\end{lemma}
\begin{proof}
The existence of unique local solutions to (\ref{e26}) follows from (\ref{e59f}) and is shown in detail in \cite[Prop. 4.4]{da2003strong}. Furthermore, a fixed point argument in a weighted supremum norm shows that
$v_\eps(T^\ast)\in \CC(\T^2)$. Since, for $\CC(\T^2)$-valued initial datum, (\ref{e26}) admits a unique global solution in $\CT \cap \CC((0,\infty),\CC^\infty(\T^2))$, see e.g.\ \cite[Thm. 6.4; Prop. 6.23]{hairer2009introduction}, the claim follows from the fact that this space is a subspace of $\ET$. 
 \end{proof} 
 
Before we state the main result of this section, we introduce the Banach space
\begin{equ}
\ET^\delta: = \CC([\delta ,T]; \besn)\cap L^p([0,T]; \bes{p}{r}{\bar s}), \qquad \delta\in [0,T),
\end{equ}
equipped with the norm $\norm{x}_{\ET^\delta}:=\norm{x}_{\CC([\delta,T]; \bes{p}{r}{s})}+ \norm{x}_{L^p([0,T];\bes{p}{r}{\bar s})}$. With this notation, we have:

\begin{proposition}\label{e58}
Assume \eref{e:assExp} and consider the sequence of regularized problems $($\ref{e26}$)$ with fixed initial condition $u^0\in \besn$. For all $T>0$, the unique global solution $v_\eps \in \ET$ from Lemma~\ref{huhu} converges to zero in  sense that,
for every $\delta\in (0,T)$ ,
$\lim_{\eps \to 0}\norm{v_\eps}_{\ET^\delta}= 0$ in probability.
\end{proposition}

\begin{proof}
We introduce the stopping time $\tau_\eps^\delta$ as
\begin{equ}\label{e59}
\tau_\eps^\delta :=T \wedge \inf \set{t\geq \delta : \norm{v_\eps}_{\mathcal{E}_t^\delta}  \geq 1},
\end{equ}
with the convention that $\tau_\eps^\delta = T$ if the set is empty.
Next, we establish the limit
\begin{align}\label{e60}
\lim_{\eps \to 0} \E \norm{v_\eps}_{\ETw^\delta} =0\;.
\end{align}
Recalling that $v_\eps$ solves the fixed point equation $\CM^\eps \, v_\eps = v_\eps$,
we can use Lemma~\ref{e57}, combined with 
\begin{align}\label{e58bb}
\sup_{t\in [\delta, T]} \norm{e^{tA_\eps}\ob{u^0-z_\eps(0)}}_{\besn} \leq   e^{-\delta\thinspace C_\eps} \norm{\ob{u^0-z_\eps(0)}}_{\besn},
\end{align}
 to show that there exists $\gamma >0$ and $K_\eps$ with $\lim_{\eps \to 0}K_\eps = 0$ such that 
\begin{align*}
\norm{v_\eps}_{\ETw^\delta}& \leq K_\eps \ob{1+ T^{\gamma}} \norm{u^0-z_\eps(0)}_{\besn}\\
&\qquad  + K_\eps  T^{\gamma} \sum_{l=0}^3 \, \norm{v_\eps}^l_{\ETw^\delta}\bigl\|\wick{z_\eps^{3-l}}\bigr\|_{L^p([0, \tau^\delta_\eps]; \besn)}. 
\end{align*}
Since $\norm{v_\eps}_{\ETw^\delta} \le 1$ by construction,
the claim \eref{e60} then follows from Lemma~\ref{lemma2} and Corollary~\ref{aha2}. 
Since, by the definition of $\tau_\eps^\delta$, this implies that $\lim_{\eps \to 0} \P(\tau_\eps^\delta <T) = 0$,
the claim follows.
\end{proof}

\begin{proof}[Proof of Theorem \ref{main_thm1}]  
Since $u_\eps = z_\eps + v_\eps$, the claim follows from Propositions~\ref{prop1} and \ref{e58}, in conjunction with the embedding
$\bes{\bar p}{r}{\bar s} \hookrightarrow \bes{p}{r}{s}$, which holds if $\bar s \ge s$ and $\bar p \ge p$.
\end{proof}

\section{Deterministic limit for weak noise}
\label{sec:det}

In this section, we give the proof of Theorem~\ref{main_thm2}. The technique of proof is almost identical to
the previous section, but we define objects in a slightly different way. This time, we define an operator $A = \Delta - 1$,
and we set 
\begin{align}\label{newsc2}
z_\eps(t):=\sigma(\eps)\int_{-\infty}^t e^{(t-s)A} dW_\eps(s)\;.
\end{align}
We furthermore define all of our Wick products with respect to the law $\mu_\eps$ of $z_\eps$, so that
all throughout this section \eref{wickpower} holds, but with $D_\eps$ given by
\begin{equ}
D_\eps^2 = \frac{1}{8\pi^2} {\sum_{\abs{k}\leq 1/\eps} \frac{\sigma^2(\eps)}{1 +\abs{k}^2}}\;.
\end{equ}
Note that, by Lemma~\ref{lem:intLog}, one has
\begin{equ}
\lim_{\eps\to0} D_\eps^2 = {\lambda^2 \over 8\pi}\;.
\end{equ}
As before, we rewrite the solution to (\ref{e103}) as $u_\eps=v_\eps + z_\eps$, where $v_\eps$ is solution to 
\begin{equ}\label{nocheinmal}
\partial_t v_\eps = A v_\eps + (2-3D_\eps^2)\bigl(v_\eps + z_\eps\bigr) + \sum_{l=0}^3 a_l v_\eps^l\, \wick{z_\eps^{3-l}} \;,
\end{equ}
with initial condition $v_\eps(0) = u^0 - z_\eps(0)$ and suitable constants $a_l$.

Note first that one has the following result:
\begin{proposition}
Let $z_\eps$ be defined as in \eref{newsc2}. Then, for every $T>0$ and every $n>0$, the limits
\begin{equ}
\lim_{\eps\to0}\norm{z_\eps}_{\CC([0,T];\besn)} = 0\;,\quad
\lim_{\eps\to0}\norm{\wick{z_\eps^n}}_{L^p([0,T];\besn)} = 0\;,
\end{equ}
hold in probability.
\end{proposition}

\begin{proof}
It follows from \cite[Lem.~3.2]{da2003strong} that 
\begin{equ}
\E \norm{\wick{z_\eps^n}}_{\besn} \lesssim \sigma^n(\eps) \to 0\;,
\end{equ}
as $\eps \to 0$. The proof that $z_\eps$ also converges to $0$ in $\CC([0,T];\besn)$ is virtually identical to the
proof of Proposition~\ref{prop1}, so we omit it.
\end{proof}

It remains to establish that
$\lim_{\eps\to0}\norm{v_\eps - w_\lambda}_{\CC([0,T];\besn)}= 0$ in probability, which is the content of the following result:

\begin{proposition}
Assume \eref{e:assExp} and let $u_0\in \besn$. Let $v_\eps\in \ET$ be the unique mild solution to  (\ref{nocheinmal}), and $w_\lambda\in \ET$  the unique solution to $($\ref{det_alt}$)$. Then 
\begin{equ}
\lim_{\eps\to0} \norm{v_\eps - w_\lambda}_{\ET} = 0
\end{equ}
in probability.
\end{proposition}
\begin{proof}
Setting $\delta_\eps = 3D_\eps^2 - {3\lambda^2 \over 8\pi}$ and $a_\lambda = 1 - {3\lambda^2 \over 8\pi}$, we can rewrite the equations for $v_\eps$ and $w_\lambda$ as
\begin{equs}
\d_t v_\eps &= \Delta v_\eps + a_\lambda v_\eps - v_\eps^3 - \delta_\eps v_\eps + (2-3D_\eps^2) z_\eps + \sum_{l=0}^2 a_l v_\eps^l \,\wick{z_\eps^{3-l}}\;,\\
\d_t w_\lambda &= \Delta w_\lambda + a_\lambda  w_\lambda - w_\lambda^3\;.
\end{equs}
Setting $\rho_\eps = v_\eps - w_\lambda$, we see that $\rho_\eps$ solves the following evolution equation:
\begin{equs}
\d_t \rho_\eps &= \bigl(\Delta + a_\lambda\bigr) \rho_\eps - \rho_\eps \bigl(v_\eps^2 + v_\eps w_\lambda + w_\lambda^2\bigr)  \\
&\quad + (2-3D_\eps^2) z_\eps -\delta_\eps v_\eps + \sum_{l=0}^2 a_l v_\eps^l \,\wick{z_\eps^{3-l}}\;.
\end{equs}
Setting $\hat A = \Delta + a_\lambda$, we have the mild formulation
\begin{equs}
\rho_\eps(t) &= e^{\hat At} \rho_\eps(0) - \int_0^t e^{\hat A(t-s)}\rho_\eps(s) \bigl(v_\eps^2 + v_\eps w_\lambda + w_\lambda^2\bigr)(s)\,ds \\
&\quad + (2-3D_\eps^2)\int_0^t e^{\hat A(t-s)}z_\eps(s)\,ds -\delta_\eps \int_0^t e^{\hat A(t-s)}v_\eps(s)\,ds \\
&\quad +   \sum_{l=0}^2 a_l \int_0^t e^{\hat A(t-s)}v_\eps^l(s) \,\wick{z_\eps^{3-l}}(s)\,ds\;.
\end{equs}
It then follows from Lemmas~\ref{lemma_imp} and \ref{cor10} that
\begin{equs}
\|\rho_\eps\|_{\ET} &\lesssim  \| \rho_\eps(0) \|_{\bes{p}{r}{s}} + T^\delta \|\rho_\eps\|_{\ET} \bigl(\|v_\eps\|_{L^p([0,T]; \bes{p}{r}{\bar s})}^2 + \|w_\lambda\|_{L^p([0,T]; \bes{p}{r}{\bar s})}^2\bigr) \\
&\quad +\delta_\eps \|v_\eps\|_{\ET} + T^\delta \sum_{l=0}^2 (1+\|v_\eps\|_{\ET}^l) \,\|\wick{z_\eps^{3-l}}\|_{L^p([0,T]; \besn)}\;.\label{e:boundrho}
\end{equs}
We now use the fact that there exists $K$ 
such that the deterministic solution $w_\lambda$ satisfies $\|w_\lambda\|_{\ET} \le K$.
Setting $\tau_\eps = \bar T \wedge \inf \{t : \norm{\rho_\eps}_{\mathcal{E}_t}  \geq 1\}$ for some $\bar T \le T$ such that
$\bar T^\delta ((K+1)^2 + K^2) \le {1\over 2}$, it follows from \eref{e:boundrho} that
\begin{equ}
\|\rho_\eps\|_{\ETw} \lesssim  \| \rho_\eps(0) \|_{\bes{p}{r}{s}} + \bar T^\delta \sum_{l=0}^2 (1+K)^l \,\bigl(\delta_\eps + \|\wick{z_\eps^{3-l}}\|_{L^p([0,T]; \besn)}\bigr)\;.
\end{equ}
This bound can easily be iterated, and the
claim then follows similarly to the proof of Proposition~\ref{e58}.
\end{proof}

\appendix

\section{Technical results}

In this appendix, we collect a few technical results. 
\begin{lemma}\label{lemma_imp}
Let $A = \Delta - \Lambda$ for $\Lambda \ge 1$ and let $f\in L^{p/n}([0,T]; \bes{p}{r}{(2n-1)s})$ with $p> n\geq1$, $s<0$ and  $\bar s=2/p+2s$ such that \begin{align}\label{e44a}
(n-1)s+1-\frac np>0.
\end{align}
Then  there exists $\delta > 0$ such that
\begin{equ}
\norm{\int_0^t e^{(t-\tau)A }f(\tau)d\tau }_{\mathcal{E}_T} \leq K(\Lambda) \,T^\delta  \,\norm{f}_{L^{p/n}([0,T]; \bes{p}{r}{(2n-1)s})},
\end{equ}
with a constant $K(\Lambda)$ such that $\lim_{\Lambda \to \infty} \, K(\Lambda) =0.$
\end{lemma}
\begin{proof}
Modulo straightforward modifications yielding $K(\Lambda)\to 0$, the proof is identical to the proof of \cite[Lem.~3.6]{da2003strong}. 
\end{proof}

\begin{lemma}\label{cor10}
Let $n, p,r \geq1$, $s<0$, $\bar s=2/p+2s$ such that $\abs{s}<\frac{2}{p(2n+1)}$ and $l < n$. Assume that $g_i\in L^p([0,T]; \bes{p}{r}{\bar s})$ for $i=1,\ldots, l$ and $h\in L^p([0,T]; \bes{p}{r}{s})$. Then, there exists a constant $C>0$ such that 
\begin{align}\label{e27o}
\norm{h\, g_1\cdots g_l}_{L^{p/(l+1)}([0,T];\bes{p}{r}{(2l+1)s})} \leq C\, \norm{h}_{L^p([0,T];\bes{p}{r}{s})}\, \prod_{j=1}^l \norm{g_j}_{L^p([0,T];\bes{p}{r}{\bar s})}.
\end{align}
\end{lemma}
\begin{proof} 
This is a straightforward modification of \cite[Cor. 3.5]{da2003strong}.
\end{proof}

\begin{lemma}\label{a14}  
 Let $p,r\geq 1$ and $\bar s < s$. Then, there exists a constant $C>0$ such that  
\begin{align*}
\norm{e^{t\Delta} x}_{\bes{p}{r}{s}} \leq C t^{\bar s - s \over 2} \, \norm{x}_{\bes{p}{r}{\bar s}} \qquad \forall x\in \bes{p}{r}{\bar s}.
\end{align*}
\end{lemma}
\begin{proof} 
The estimate follows from \cite[Lem.~2.4]{bahouri2010fourier} and the definition of the Besov norm (\ref{e3abc}).
\end{proof}

\begin{corollary}\label{IC_cor}
Let $s<0$, $r, p\geq 1$, and $\bar s=2s+\frac2p$. Define the operator $A = \Delta - \Lambda$ and recall the $\ET$-norm as defined in (\ref{max_norm}). Then there exists $\delta>0$ such that for all $\Lambda>1$,
\begin{align*}
\norm{e^{tA} x}_{\ET} \leq \ob{1+C(\Lambda)\, T^{\delta}} \, \norm{x}_{\besn},\qquad \forall x\in \besn,
\end{align*}
where $\lim_{\Lambda\to\infty} C(\Lambda)=0$. 
\end{corollary}
\begin{proof}
The bound on the $C\ob{\cb{0,T}; \besn}$ norm is trivial. Using Proposition \ref{a14}, we obtain for arbitrary $\gamma>0$
\begin{align*}
\norm{e^{tA}x}_{L^p([0,T];\besn)} &\leq \frac{K}{\Lambda^{\gamma/p}}\ob{\int_0^T \frac{1}{t^\gamma} \norm{e^{t\Delta}x}_{\besn}dt}^{1/p}\\
&\leq \frac{K}{\Lambda^{\gamma/p}} \norm{x}_{\besn} \ob{\int_0^T t^{p(s-\bar s)/2-\gamma}dt}^{1/p}\\
& \leq \frac{K}{\Lambda^{\gamma/p}}\norm{x}_{\besn}T^{\abs{s}/2 -\gamma /p}.
\end{align*}
Choosing $\gamma<\frac{p}{2}\abs{s}$, the claim follows.
\end{proof}

\endappendix

\bibliography{./PhD11_database}

\begin{thebibliography}{KORVE07}
\expandafter\ifx\csname url\endcsname\relax
  \def\url#1{\texttt{#1}}\fi
\expandafter\ifx\csname urlprefix\endcsname\relax\def\urlprefix{URL }\fi

\bibitem[AC79]{allen1979microscopic}
\textsc{S.~Allen} and \textsc{J.~Cahn}.
\newblock {A microscopic theory for antiphase boundary motion and its
  application to antiphase domain coarsening}.
\newblock \emph{Acta Metall.} \textbf{27}, no.~6, (1979), 1085--1095.

\bibitem[AR91]{AlbRock91}
\textsc{S.~Albeverio} and \textsc{M.~R{\"o}ckner}.
\newblock Stochastic differential equations in infinite dimensions: solutions
  via {D}irichlet forms.
\newblock \emph{Probab. Theory Related Fields} \textbf{89}, no.~3, (1991),
  347--386.

\bibitem[BBM10]{BovierAC10}
\textsc{F.~Barret}, \textsc{A.~Bovier}, and \textsc{S.~M{\'e}l{\'e}ard}.
\newblock Uniform estimates for metastable transition times in a coupled
  bistable system.
\newblock \emph{Electron. J. Probab.} \textbf{15}, (2010), no. 12, 323--345.

\bibitem[BCD10]{bahouri2010fourier}
\textsc{H.~Bahouri}, \textsc{J.~Chemin}, and \textsc{R.~Danchin}.
\newblock \emph{Fourier analysis and nonlinear partial differential equations},
  vol. 343 of \emph{Grundlehren der mathematischen Wissenschaften Series}.
\newblock Springer Verlag, 2010.

\bibitem[CF11]{Cerr11}
\textsc{S.~Cerrai} and \textsc{M.~Freidlin}.
\newblock Approximation of quasi-potentials and exit problems for
  multidimensional {RDE}'s with noise.
\newblock \emph{Trans. Amer. Math. Soc.} \textbf{363}, no.~7, (2011),
  3853--3892.

\bibitem[DPD03]{da2003strong}
\textsc{G.~Da~Prato} and \textsc{A.~Debussche}.
\newblock {Strong solutions to the stochastic quantization equations}.
\newblock \emph{Ann. Probab.} \textbf{31}, no.~4, (2003), 1900--1916.

\bibitem[DPT07]{da2007wick}
\textsc{G.~Da~Prato} and \textsc{L.~Tubaro}.
\newblock {Wick powers in stochastic PDEs: an introduction}.
\newblock \emph{Technical Report UTM 711, University of Trento} (2007).

\bibitem[DPZ92]{da1992stochastic}
\textsc{G.~Da~Prato} and \textsc{J.~Zabczyk}.
\newblock \emph{{Stochastic equations in infinite dimensions}}, vol.~45 of
  \emph{Encyclopedia of mathematics and its applications}.
\newblock Cambridge University Press, 1992.

\bibitem[ESS92]{MeanCurv}
\textsc{L.~C. Evans}, \textsc{H.~M. Soner}, and \textsc{P.~E. Souganidis}.
\newblock Phase transitions and generalized motion by mean curvature.
\newblock \emph{Comm. Pure Appl. Math.} \textbf{45}, no.~9, (1992), 1097--1123.

\bibitem[GJ87]{GliJaf87}
\textsc{J.~Glimm} and \textsc{A.~Jaffe}.
\newblock \emph{Quantum physics}.
\newblock Springer-Verlag, New York, second ed., 1987.
\newblock A functional integral point of view.

\bibitem[Hai09]{hairer2009introduction}
\textsc{M.~Hairer}.
\newblock {An introduction to stochastic PDEs}, 2009.
\newblock \urlprefix\url{http://www.hairer.org/Teaching.html}.
\newblock Unpublished lecture notes.

\bibitem[Ilm93]{MeanCurv2}
\textsc{T.~Ilmanen}.
\newblock Convergence of the {A}llen-{C}ahn equation to {B}rakke's motion by
  mean curvature.
\newblock \emph{J. Differential Geom.} \textbf{38}, no.~2, (1993), 417--461.

\bibitem[KORVE07]{kohn2007action}
\textsc{R.~Kohn}, \textsc{F.~Otto}, \textsc{M.~Reznikoff}, and
  \textsc{E.~Vanden-Eijnden}.
\newblock {Action minimization and sharp-interface limits for the stochastic
  Allen-Cahn equation}.
\newblock \emph{Commun. Pure Appl. Math.} \textbf{60}, no.~3, (2007), 393--438.

\bibitem[LG50]{GL}
\textsc{L.~Landau} and \textsc{L.~Ginzburg}.
\newblock On the theory of superconductivity.
\newblock \emph{J. Expt. Theor. Phys.} \textbf{20}, (1950), 1064--1082.

\bibitem[PW81]{ParisiWu81}
\textsc{G.~Parisi} and \textsc{Y.~S. Wu}.
\newblock Perturbation theory without gauge fixing.
\newblock \emph{Sci. Sinica} \textbf{24}, no.~4, (1981), 483--496.

\bibitem[RNT11]{ryser2011well}
\textsc{M.~Ryser}, \textsc{N.~Nigam}, and \textsc{P.~Tupper}.
\newblock On the well-posedness of the stochastic {A}llen-{C}ahn equation in
  two dimensions.
\newblock \emph{J. Comp. Phys.} (2011).
\newblock Accepted.

\bibitem[Wal86]{walsh1986introduction}
\textsc{J.~Walsh}.
\newblock {An introduction to stochastic partial differential equations}.
\newblock \emph{{\'E}cole d'{\'E}t{\'e} de Probabilit{\'e}s de Saint Flour
  XIV-1984}  265--439.

\end{thebibliography}
	\bibliographystyle{./Martin}

\end{document}